\documentclass[12pt]{amsart}
\usepackage{amsmath,amssymb,mathtools,mathabx,rotating}
\usepackage{graphicx}
\usepackage{color}
\usepackage{epstopdf}
\newcommand*{\id}{{\mathrm{id}}}

\newcommand*{\un}{{\mathbf 1}}

\def\shuff#1#2{\mathbin{
      \hbox{\vbox{\hbox{\vrule \hskip#2 \vrule height#1 width 0pt}\hrule}\vbox{\hbox{\vrule \hskip#2 \vrule height#1 width 0pt\vrule }\hrule}}}}
\def\shuffl{{\mathchoice{\shuff{5pt}{3.5pt}}{\shuff{5pt}{3.5pt}}{\shuff{3pt}{2.6pt}}{\shuff{3pt}{2.6pt}}}}
\def\shuffle{{\, \shuffl \,}}

\newtheorem{thm}{Theorem}

\newtheorem{cor}[thm]{Corollary}

\newtheorem{lem}[thm]{Lemma}
\newtheorem{prop}[thm]{Proposition}
\newtheorem{defn}{Definition}

\begin{document}
 
\title[Cumulants and shuffles]{Monotone, free, and boolean cumulants:\\[0.1cm] a shuffle algebra approach.}

\vspace{1cm}

\author{Kurusch Ebrahimi-Fard}
\address{Department of Mathematical Sciences, 
		Norwegian University of Science and Technology (NTNU), 
		7491 Trondheim, Norway.
		{\tiny{On leave from UHA, Mulhouse, France.}}}
\email{kurusch.ebrahimi-fard@ntnu.no, kurusch.ebrahimi-fard@uha.fr}         
\urladdr{https://folk.ntnu.no/kurusche/}

\author{Fr\'ed\'eric Patras}
\address{Universit\'e C\^ote d'Azur,
		Labo.~J.-A.~Dieudonn\'e,
         		UMR 7351, CNRS,
         		Parc Valrose,
         		06108 Nice Cedex 02, France.}
\email{patras@unice.fr}
\urladdr{www-math.unice.fr/$\sim$patras}

\begin{abstract}
The theory of cumulants is revisited in the ``Rota way'', that is, by following a combinatorial Hopf algebra approach. Monotone, free, and boolean cumulants are considered as infinitesimal characters over a particular combinatorial Hopf algebra. The latter is neither commutative nor cocommutative, and has an underlying unshuffle bialgebra structure which gives rise to a shuffle product on its graded dual. The moment-cumulant relations are encoded in terms of shuffle and half-shuffle exponentials. It is then shown how to express concisely monotone, free, and boolean cumulants in terms of each other using the pre-Lie Magnus expansion together with shuffle and half-shuffle logarithms.
\end{abstract}

\maketitle

\begin{quote}
\end{quote}


\section{Introduction}
\label{sec:intro}

Since the discovery of free cumulants as the proper way of encoding the notion of free independence in Voiculescu's theory of free probability \cite{voiculescu_92,voiculescu_95}, together with their combinatorial description in terms of non-crossing partitions (see e.g.~\cite{nicaspeicher_06}, also for further references), the literature on the interplay between the various notions of independence and cumulants has flourished and is still flourishing, as illustrated by recent works such as \cite{lehner-etal_15,eps_17,hasebesaigo_11,josuat_13,josuatetal_16,sw_2016}.

In \cite{ebrahimipatras_15,ebrahimipatras_16} we proposed a new approach to free cumulants. Consider a non-commutative probability space $(A,\phi)$ with its unital state $\phi: A \to k$. The central result in \cite{ebrahimipatras_15} is a concise description of the relation between moments and cumulants in free probability in terms of a fixed point equation defined on the graded dual of $T(T(A))$, the double tensor algebra over $A$ equipped with a suitable Hopf algebra structure, which actually is an unshuffle (or codendriform) bialgebra structure. This fixed point equation is solved by using a left half-shuffle exponential. The latter may be considered a natural extension of the notion of time-ordered exponential (also known as Picard or Dyson expansion) familiar in the context of linear differential equations.

In the present work we extend the results in \cite{ebrahimipatras_15} by showing that  $T(T(A))$ also encodes in a rather simple manner monotone as well as boolean cumulants. The key to this construction is provided by two other exponential maps (the shuffle and right-half shuffle exponentials), which are naturally defined on the dual of $T(T(A))$. For example, the relation between moments and multivariate monotone cumulants, usually obtained in terms of monotone partitions \cite{hasebesaigo_11}, is recovered by evaluating the shuffle exponential on $T(A)$. The three exponentials allow to compute moments in terms of the corresponding cumulants, and have logarithmic inverses that in return permit to expand cumulants in terms of moments.

From general algebraic arguments, one deduces an intriguing link between the half-shuffle exponentials and the shuffle exponential, given in terms of the pre-Lie Magnus expansion and its compositional inverse. This yields a simple way to express monotone, free, and boolean cumulants concisely in terms of each other. For example, the pre-Lie Magnus expansion permits to express bijectively free cumulants in terms of monotone cumulants. This extends consistently to new algebraic relations among the infinitesimal characters corresponding to monotone, free, and boolean cumulants (Theorem \ref{thm:transforming}). We show, among others, how to recover from this pre-Lie algebra approach the multivariate formulas describing relations between monotone, free, and boolean cumulants in terms of (irreducible) non-crossing partitions presented in reference \cite{lehner-etal_15}. 

\medskip 

The paper is organized as follows. In the next section we show how boolean, free and monotone cumulants can be described as infinitesimal characters on the double tensor algebra defined over a non-commutative probability space. The following section introduces general notions and identities for shuffle (also known as dendriform) algebras and unshuffle (or codendriform) bialgebras. The last section shows how these identities specialize to (abstract and combinatorial) formulas relating the various notions of cumulants.

\vspace{0.4cm}

\noindent {\bf{Acknowledgements}}: The second author acknowledges support from the grant ANR-12-BS01-0017, ``Combinatoire Alg\'ebrique, R\'esurgence, Moules et Applications". Support by the CNRS GDR ``Renormalisation" and the PICS program CNRS/CSIC, LJAD-ICMAT ``Alg\`ebres de Hopf combinatoires et probabilit\'es non commutatives" is also acknowledged.


\section{Cumulants as infinitesimal characters}
\label{sec:cumulants}

In the following we denote by $k$ the ground field of characteristic zero over which all algebraic structures are defined. Furthermore, it is assumed  that any $k$-algebra $A$ is associative and unital, if not stated otherwise. The product and unit in $A$ are denoted by $a \cdot_A b:=m_A(a \otimes b)$ respectively $1_A$. The extension of the former to a map from $A^{\otimes n}$ to $A$ is denoted $m_A^{[n-1]}$. We write $T(A)=\bigoplus_{n>0}A^{\otimes n}$ (resp.~$\overline T(A)=\bigoplus_{n\geq 0}A^{\otimes n}$), for the non-unital (resp.~unital) tensor algebra over $A$. Tensors are written using the word notation, i.e., we write $a_1\cdots a_n$ for $a_1\otimes \cdots \otimes a_n$. The symbol $a_1 \cdots a_n$ should not be confused with the product of the $a_i$ seen as elements in the algebra $A$, which is written $a_1\cdot_A a_2 \cdot_A\ \cdots\ \cdot_A a_n$. In this respect, $a^n$ stands for the tensor $a^{\otimes n}$. The identity map of an object $X$ is written $\id_X$, or simply $\id$ when no confusion can occur.  

\smallskip

In \cite{ebrahimipatras_15,ebrahimipatras_16} we considered the moment-cumulant relations in free probability from a Hopf algebraic point of view. The central object is the double tensor algebra $T(T(A))$ defined over a non-commutative probability space $(A,\phi)$, where $A$ is an algebra with unit $1_A$ and the unital map $\phi$ is a $k$-valued linear form on ${A}$, such that $\phi(1_{A}) = 1$. 

The double tensor algebra is defined to be $T(T(A)):=\oplus_{n > 0} T(A)^{\otimes n}$. We use the bar-notation to denote elements $w_1 | \cdots | w_n \in T(T(A))$, where $w_i$ are words in the tensor algebra $T(A)$. The vector space $T(T(A))$ is equipped with the concatenation product $a|b:= w_1 | \cdots | w_n | w_1' | \cdots | w_m'$ for $a= w_1 | \cdots | w_n$ and $b=  w_1' | \cdots | w_m'$ in $T(T(A))$. The resulting algebra is non-commutative and both multigraded, $T(T(A))_{n_1,\ldots ,n_k}:=T_{n_1}(A)\otimes \cdots \otimes T_{n_k}(A)$, as well as graded, $T(T(A))_n:=\bigoplus_{n_1+ \cdots +n_k=n}T(T(A))_{n_1,\ldots ,n_k}$. Similar observations hold for the unital case, $\overline T(T(A))$. We will identify without further comments a bar symbol such as $w_1|1|w_2$ with $w_1|w_2$ (formally, using the canonical map from $\overline T(\overline T(A))$ to $\overline T(T(A))$). For notational clarity, we will from now on denote by $\mathbf 1$ the unit of $\overline T(T(A))$. 

Given two (canonically ordered) subsets $S \subseteq U$ of the set of integers $\bf N^\ast$, we call connected component of $S$ relative to $U$ a maximal sequence $s_1 < \cdots < s_n$ in $S$ such that there are no $ 1 \leq i < n$ and $u \in U$, such that $s_i < u < s_{i+1}$. In particular, a connected component, or interval of $S$ in $\bf N$ is simply a maximal sequence of successive elements $s,s+1,\ldots ,s+n$ in $S$. Consider a word $a_1 \cdots a_n \in T(A)$. For a nonempty set $S:=\{s_1< \cdots < s_p\} \subseteq [n]$, we define $a_S:= a_{s_1} \cdots a_{s_p} \in T(A)$, and $a_\emptyset:=\un$. Denoting by $J_1,\ldots,J_l$ the connected components of $[n] - S$, we also set $a_{J^S_{[n]}}:= a_{J_1} | \cdots | a_{J_l} \in T(T(A))$. More generally, for $S \subseteq U \subseteq [n]$, set  $a_{J^S_U}:= a_{J_1} | \cdots | a_{J_l}\in T(T(A))$, where the ${J_j}$ are now the connected components of $U-S$ in $U$. 

\begin{thm} \cite{ebrahimipatras_15} \label{thm:HA}
The connected graded algebra $H:=\overline T(T(A))$ with the coproduct $\Delta : T(A) \to \overline T(A) \otimes  \overline T(T(A))$ defined by 
\begin{equation}
\label{HA}
	\Delta(a_1\cdots a_n) := \sum_{S \subseteq [n]} a_S \otimes  a_{J_1} | \cdots | a_{J_k},
\end{equation} 
and extended multiplicatively to all of $\overline T(T(A))$ (with $\Delta(\un):= \un \otimes \un$), is a connected graded non-commutative and non-cocommutative Hopf algebra with counit and unit the canonical maps $\epsilon : \overline T(T(A)) \to k=\overline T(T(A))_0$ and $\eta : k=\overline T(T(A))_0 \to H$. 
\end{thm}

We write $H^+=T(T(A))$ and $u:=\eta\circ \epsilon$.
The antipode $S$ of $H$ is the inverse of the identity  map  $\id$ with respect to the convolution product, and given through
\begin{equation}
\label{antipode}
	S=\frac{1}{u + P}=\sum_{i \ge 0} (-1)^i P^{\ast i},
\end{equation}
where $P$ is the linear map $P:= \id - u$. Note that $P(\un)=0$ and $P=\id$ on $H^+$, the kernel of the counit. 

Recall that the convolution product on the space of linear maps, $Lin(H,H)$ (resp.~$Lin(H,k)$), is given by $f * g := m_H \circ (f \otimes g) \circ \Delta$ (resp.~$f * g := m_k \circ (f \otimes g) \circ \Delta$) for $f,g \in Lin(H,H)$ (resp.~$f,g \in Lin(H,k)$). It defines a unital $k$-algebra structure on $Lin(H,H)$ (resp.~${{\mathcal L}_A}:=Lin(H,k)$) with $u$ (resp.~the counit $\epsilon: H \to k$) as its unit.

\begin{defn}
A linear form $\Phi \in {{\mathcal L}_A}$ is a character if it is unital, $\Phi(\un)=1$, and multiplicative, i.e., for $w_1,w_2 \in H$, $\Phi(w_1|w_2)=\Phi(w_1)\Phi(w_2).$ A linear form $\kappa \in {{\mathcal L}_A}$ is an infinitesimal character if $\kappa(\un)=0$ and if for $w_1,w_2 \in H^+$, $\kappa(w_1|w_2)=0.$
\end{defn}

Recall that the set $G(A) \subset {\mathcal L}_A$ of characters forms a group with respect to the convolution product. The reader is refered to \cite{cartier_07,figueroagraciabondia_05,manchon_08} for details. The inverse of an element $\Phi \in G(A)$ is given by composition with the Hopf algebra antipode, $\Phi^{-1}=\Phi \circ S$. The set $g(A) \subset {\mathcal L}_A$ of infinitesimal characters forms a Lie algebra with Lie bracket defined by the commutator in ${\mathcal L}_A$. For $\alpha \in g(A)$ and any word $w \in T(A)$ the exponential $\exp^*(\alpha)(w) := \sum_{j\ge 0} \frac{1}{j!}\alpha^{* j}(w)$ reduces to a finite sum. It is well-known that $\exp^*$ restricts to a natural bijection from $g(A)$ onto the group $G(A)$. The inverse of $\exp^*$ is given by $\log^*(\epsilon +\gamma)(w)=\sum_{l\ge 1}{(-1)^{l-1}\over l}\gamma^{* l}(w)$, where the sum terminates after a finite number of terms. 

Let $(A,\phi)$ be a non-commutative probability space. Its linear form $\phi$ is first extended to $T(A)$ by defining for all words $a_1\cdots a_n \in A^{\otimes n}$ 
$$
	\phi(a_1 a_2 a_3 \cdots a_n) := \phi(a_1\cdot_A a_2\cdot_A  a_3\cdot_A\ \cdots\  \cdot_A a_n).
$$
This map $\phi$ is then extended multiplicatively to a map $\Phi: \overline T(T(A)) \to k$ with $\Phi(\un):=1$ and 
$$
	{\Phi}(w_1 | \cdots | w_m) := \phi(w_1) \cdots \phi(w_m). 
$$
Viewing the letters in the word $w=a_1 \cdots a_n \in T(A)$ as non-commutative random variables, the multivariate moment of $w$ is defined to be $m_n(a_1, \ldots, a_n):=\phi(a_1\cdot_A a_2\cdot_A  a_3\cdot_A\ \cdots\  \cdot_A a_n)=\phi(w)$.


\subsection{Monotone cumulants as infinitesimal characters}
\label{ssec:monotonecumulants}

\begin{thm} \label{tim:monotone}
Let $(A,\phi)$ be a non-commutative probability space with unital map $\phi: A \to k$ and $\Phi$ its extension to $\overline T(T(A))$ as a character. Let the map $\rho: \overline T(T(A))\to k$ be the infinitesimal character solving $\Phi = \exp^*(\rho)$, i.e., $\rho = \log^*(\Phi)$. For $a \in A$ we set $h_n:=\rho(a^{n})$, $n \geq 1$, and $m_l:=\Phi(a^{l})=\phi(a^l)$, $l\geq 0$. Then 
\begin{equation}
\label{monotone}
	m_n=\sum\limits_{s=1}^n\sum\limits_{1=i_0 < \cdots < i_{s-1} < i_s = n+1} \frac{1}{s!}
	\prod_{j=1}^s i_{j-1} h_{i_j - i_{j-1}}.
\end{equation}
In particular, the $h_n$ identify with the monotone cumulants.
\end{thm}

The univariate monotone moment-cumulant relation \eqref{monotone} was given by Hasebe and Saigo in \cite{hasebesaigo_11}. Before proving the statement of Theorem \ref{tim:monotone}, let us calculate $\Phi(a^n) = \exp^*(\rho)(a^n)$, $a^n \in T(A)$, for $n=1,2,3,4$. 
\begin{eqnarray*}
	\Phi(a) 	&=& \exp^*(\rho)(a)=\rho(a) = h_1\\
	\Phi(aa) 	&=& \big(\rho + \frac{1}{2!}\rho * \rho\big) (aa) 
			  = \rho (aa) +  \frac{1}{2!} m_k\circ(\rho \otimes \rho)\circ\Delta(aa)  = h_2 + h_1h_1\\
	\Phi(aaa) 	&=& \big(\rho + \frac{1}{2!}\rho * \rho+ \frac{1}{3!}\rho * \rho* \rho\big) (aaa) 
			   = h_3 + \frac{5}{2}h_1h_2 + h_1h_1h_1\\		
	\Phi(aaaa) 	&=& \big(\rho + \frac{1}{2!}\rho * \rho+ \frac{1}{3!}\rho * \rho* \rho + \frac{1}{4!}\rho *\rho * \rho* \rho\big) (aaaa) \\
			&=& h_4 
				+ 3 h_1h_3 
			         +\frac{3}{2}h_2h_2 
			         +\frac{13}{3}h_1h_1h_2 
			         + h_1h_1h_1h_1.\\				   
\end{eqnarray*}
\begin{proof} (Thm.~\ref{tim:monotone})
Recall that $\rho \in g(A)$ is an infinitesimal character, i.e., $\rho(\un)=0$ and $\rho(w_1|w_2)=0$ for any words $w_1,w_2 \in T(A)$. Then $\exp^*(\rho)=\sum_{i=0}^\infty \frac{1}{i!} \rho^{* i}$ is such that
$$
	m_l=\exp^*(\rho)(a^l)= \sum_{i=0}^l \frac{1}{i!} \rho^{* i}(a^l).
$$
Now recall the definition of the convolution product $\rho_1 * \rho_2 * \cdots * \rho_l = m^{[l-1]}_k(\rho_1 \otimes \rho_2 \otimes \cdots \otimes \rho_l)\Delta^{[l-1]}$, where $m^{[0]}_k=\Delta^{[0]}:=id$, $\Delta^{[1]}= \Delta$, $\Delta^{[l]} := (\Delta^{[l-1]} \otimes \id)\Delta$, and the coproduct $\Delta$ was defined in \eqref{HA}
$$
	\Delta(a_1\cdots a_n) = \sum_{S \subseteq [n]} a_S \otimes a_{J^S_{[n]}}.
$$
Since $\rho \in g(A)$ we are only interested in what we will call the ``reduced linearised'' part of $\Delta$, i.e., the composite, written $\overline\Delta$, of $\Delta$ with the canonical projection  to $T(A) \otimes T(A)$. In particular
\begin{equation}
\label{reducedcoprod1}
	\overline\Delta(a^{l}) := \sum_{I \subset [l],\ I \neq \emptyset \atop I\, \rm{interval}} a_{[l-|I|]} \otimes  a_I
	= \sum_{k=1}^{l-1} (k+1) a^{k} \otimes a^{l-k}.
\end{equation}
Iteration on the left gives 
$$
	(\overline\Delta \otimes \id)\overline\Delta(a^{l}) 
	= \sum_{k_1=1}^{l-1}\sum_{k_2=1}^{k_1-1} (k_1+1) (k_2+1) a^{k_2} \otimes a^{k_1-k_2} \otimes a^{l-k_1}
$$
and
\begin{align*}
	\overline\Delta^{[n]}(a^{l}) &= (\overline\Delta^{[n-1]} \otimes \id ) \overline\Delta  (a^{l}) \\
	&= \sum_{k_1=1}^{l-1}\sum_{k_2=1}^{k_1-1} \cdots \sum_{k_{n-1}=1}^{k_{n-2}-1}(k_1+1) (k_2+1) \cdots (k_{n-1}+1) 
	a^{k_{n-1}} \otimes  a^{k_{n-2}-k_{n-1}} \otimes \cdots\\
	&\hspace{7cm} \otimes  a^{k_3 - k_2} \otimes a^{k_1-k_2} \otimes a^{l-k_1},
\end{align*}
so that $\overline\Delta^{[n-1]}(a^{n}) = n! a^{\otimes n}.$

The rationale behind the reduced linearised coproduct $\overline\Delta$ is the fact that $\rho \in g(A)$ maps $\un$ as well as any expression $w_1|\cdots |w_j \in H$, $w_1,\ldots,w_j \in T(A)$ to zero. It is important to note that the reduced linearised coproduct is not coassociative, and the above only holds in the case of left-iteration $\overline\Delta^{[n]} := (\overline\Delta^{[n-1]} \otimes \id ) \overline\Delta$. Therefore 
\begin{align*}
	 \frac{1}{n!} \rho^{*n}(a^l) &= \frac{1}{n!} 
	 \sum_{k_1=1}^{l-1}\sum_{k_2=1}^{k_1-1} \cdots \sum_{k_{n-1}=1}^{k_{n-2}-1}  (k_1+1) (k_2+1) \cdots (k_{n-1}+1) \\
	&\hspace{5cm}\rho(a^{k_{n-1}}) \rho(a^{k_{n-2}-k_{n-1}}) \cdots \rho(a^{k_3 - k_2}) \rho(a^{k_1-k_2}) \rho(a^{l-k_1})\\
	&=\frac{1}{n!} \sum_{k_1=2}^{l}\sum_{k_2=2}^{k_1-1} \cdots \sum_{k_{n-1}=2}^{k_{n-2}-1}  k_1k_2 \cdots k_{n-1} 
	\rho(a^{k_{n-1}-1}) \rho(a^{k_{n-2}-k_{n-1}}) \cdots \rho(a^{k_3 - k_2}) \rho(a^{k_1-k_2}) \rho(a^{l-k_1+1})\\
	&= \frac{1}{n!} \sum_{1=k_n < k_{n-1} < \cdots < k_1 < k_0=l+1} 
	k_1k_2 \cdots k_{n-1} h_{k_{n-1}-1} h_{k_{n-2}-k_{n-1}} \cdots h_{k_3 - k_2}h_{k_1-k_2} h_{l-k_1+1}\\
	&= \frac{1}{n!} \sum_{1=k_n < k_{n-1} < \cdots < k_1 < k_0=l+1} 
		\prod_{j=1}^{n} k_{j} h_{k_{j-1} - k_{j}},
\end{align*}
which yields \eqref{monotone} modulo some reindexing. 
\end{proof}

We consider now the multivariate case, that we have chosen to treat separately for didactical reason (since the univariate case follows from the above simple and elementary calculation that does not require the use of monotone partitions). For $w=a_1 \cdots a_n \in T(A)$ we set $h_n(a_1,\ldots,a_n)=\rho(w)$, and $m_n(a_1,\ldots,a_n):=\Phi(w)$. We will show that $\Phi = \exp^*(\rho)$ applied to a word $w=a_1 \cdots a_n \in T(A)$ of length $n$ can be given in terms of monotone partitions, or equivalently, in terms of non-crossing partitions weighted by inverse tree factorials \cite{lehner-etal_15}
\begin{equation}
\label{Lehner1}
	\Phi (w) = \exp^*(\rho)(w) = \sum_{\pi \in \mathcal{NC}_n}\frac{1}{\tau(\pi)!} h_\pi(a_1, \ldots, a_n),
\end{equation}
where $\mathcal{NC}_n$ stands for the set of non-crossing partitions and $h_\pi(a_1, \ldots, a_n)$ is defined as follows
$$
	h_\pi(a_1, \ldots, a_n):=\prod_{\pi_i \in \pi} h_{|\pi_i|}(a_{\pi_i}).
$$
Recall that a partition $\pi=\pi_1\sqcup\dots\sqcup\pi_k$ of $[n]$ (or of any finite subset of the integers) is non-crossing if and only if there are no $(i,l)\in \pi_a^2$ and $(j,m)\in \pi_b^2$, $1\leq a,b\leq k$ with $a\not= b$ and $i<j<l<m$.
Here $|\pi_i|$ denotes the number of elements in the block $\pi_i \in \pi \in\mathcal{NC}_n$. It is assumed that the elements in the block $\pi_i \in \pi$ are ordered, $\pi_i:=\{j_1 < \cdots < j_{|\pi_i|}\}$, and $a_{\pi_i}:=a_{j_1} \cdots a_{j_{|\pi_i|}}$. The tree factorial $\tau(\pi)!$ corresponds to the forest $\tau(\pi)$ of rooted trees encoding the nesting structure of the noncrossing partition $\pi \in \mathcal{NC}_n$. See \cite{lehner-etal_15} for details. 

The reduced linearised coproduct \eqref{reducedcoprod1} for a general word  $w=a_1 \cdots a_n \in T(A)$ can be expressed in terms of interval partitions
\begin{equation}
\label{reducedcoprod2}
	\overline\Delta(a_1 \cdots a_n) 
	= \sum_{I_1 \sqcup I_2 \sqcup I_3 = [n] \atop (I_2 \neq \emptyset) 
	\wedge  (I_1 \sqcup I_3 \neq \emptyset)} a_{I_1I_3} \otimes  a_{I_2},
\end{equation}
where $I_1 \sqcup I_2 \sqcup I_3 = [n]$ denotes a partition of $[n]$ into intervals $I_1=\{1,2,\ldots, i\}$, $I_2=\{i+1,i+2\ldots,j\}$, $I_3=\{j+1,j+2,\ldots,n\}$. Observe that by definition $\overline\Delta(w)=0$ for words of length one and zero. 

\begin{lem}\label{lem:monotone1}
Let us write $\overline\Delta^{[q]}: T(A) \to T(A)^{\otimes q+1}$ for the $q$-fold left-iterated reduced linearised coproduct \eqref{reducedcoprod2}, where $\overline\Delta^{[q]}:=(\overline\Delta^{[q-1]} \otimes \id)\overline\Delta$, $\overline\Delta^{[0]}:=\id$. Then we have for $w=a_1 \cdots a_n \in T(A)$ that 
\begin{equation}
 	\overline\Delta^{[q-1]}(w) = \sum_{\gamma \in \mathcal{M}_n^q \atop \gamma = \gamma_1 \sqcup \cdots \sqcup  \gamma_q} a_{\gamma_1} \otimes \cdots \otimes a_{\gamma_q}, 
\end{equation} 
where $\mathcal{M}_n^q$ is the set of monotone  partitions of the set $[n]$ into $q$ block. 
\end{lem}

Recall that the blocks of a non-crossing partition $\gamma = \gamma_1 \sqcup \cdots \sqcup  \gamma_q$ are naturally (pre)ordered. The preorder is defined by $\gamma_i \geq \gamma_j$ if and only if there exist $a,b \in \gamma_j,$ $c\in \gamma_i$ such that $a \leq c \leq b$. A monotone partition is a non-crossing partition equipped with a total order of its blocks refining the natural partial order just defined. Choosing such a total order (that we will call also a monotone labelling of the blocks) amounts to reindexing the blocks in such a way that $i<j$ implies $\gamma_i\leq\gamma_j$. Notice that if $\gamma$ is a monotone partition, $\gamma_m$ is necessarily an interval. Monotone partitions of subsets of $[n]$ are defined similarly. We refer the reader to \cite{lehner-etal_15} for details on monotone partitions. 

The Lemma follows by induction on the number of blocks. We have
$$\overline\Delta^{[q-1]}(w) = \sum\limits_\gamma \overline\Delta^{[q-2]}(a_{[n]-\gamma_q}) \otimes a_{\gamma_q},$$
where $\gamma_q$ runs over the intervals of $[n]$ which are non empty and different from $[n]$.
By induction, we have
 $$\Delta^{[q-2]}(a_{[n]-\gamma_q})=\sum\limits_{\gamma'\in\mathcal{M}_{[n]-\gamma_q}^{q-1}}a_{\gamma'_1}\otimes\dots\otimes a_{\gamma'_{q-1}}$$
 where $\mathcal{M}_{[n]-\gamma_q}^{q-1}$ stands for the set of monotone partitions of the set $[n]-\gamma_q$ into $q-1$ blocks. Since $\gamma_q$ is an interval, it follows that $\gamma=\gamma'_1\sqcup\cdots\sqcup\gamma'_{q-1}\sqcup\gamma_q$ is a monotone partition. We let the reader check that any monotone partition can be obtained in that way.


The result of Lemma \ref{lem:monotone1} is consistent with the characterisation of non-crossing partitions as set partitions which can be reduced to the empty partition by successively extracting interval blocks without crossing other blocks, see e.g.~\cite{eps_17} for an application of this idea to the study of mixtures of free and classcial probability.     

\begin{thm}
Let $(A,\phi)$ be a non-commutative probability space with unital map $\phi: A \to k$ and $\Phi$ its extension to $\overline T(T(A))$ as a character. Let the map $\rho: \overline T(T(A))\to k$ be the infinitesimal character solving $\Phi = \exp^*(\rho)$, i.e., $\rho = \log^*(\Phi)$. For the word $w=a_1 \cdots a_n \in T(A)$ we set $h_n(a_1,\ldots,a_n)=\rho(w)$, and $m_n(a_1,\ldots,a_n):=\Phi(w)$. Then
\begin{equation}
\label{monotone2}
	\Phi(w) =\sum\limits_{s=1}^n \frac{1}{s!}\sum\limits_{\gamma \in \mathcal{NC}_n^s} 
	\mathrm{m}(\gamma) h_{\gamma}(a_1, \ldots, a_n ),
\end{equation}
where $\mathrm{m}(\gamma)$ is the number of monotone labellings of $\gamma \in \mathcal{NC}_n$. In particular, $h_n(a_1, \ldots, a_n )$ identifies with the $n$-th multivariate monotone cumulant map.
\end{thm}

From the well-known fact that $\mathrm{m}(\gamma) = \frac{s!}{\tau(\gamma)!}$ for $\gamma \in \mathcal{NC}_n^s$ (the set of non-crossing partitions of $[n]$ with $s$ blocks), it follows that
\begin{equation}
\label{monotone3}
	\Phi(w) =\sum\limits_{\gamma \in \mathcal{NC}_n} \frac{1}{\tau(\gamma)!} h_{\gamma}(a_1, \ldots, a_n).
\end{equation}
These two equations (\ref{monotone2}) and (\ref{monotone3}) appeared in \cite{lehner-etal_15}. They were shown to characterize multivariate monotone cumulants (using the definition of the $h_n$ in terms of monotone partitions by using M\"obius inversion techniques). From this the last part of the theorem follows.

The insight we want to convey here is that the shuffle product and the shuffle exponential $\exp^*$ provide an alternative, compact and algebraic description of such lattice-type sums.

Conversely, the shuffle logarithm permits to express monotone cumulants $\rho = \log^*(\Phi)$ in terms of moments, i.e., for the word $w=a_1 \cdots a_n \in T(A)$ we find 
\begin{equation}
\label{monotone4}
	h_n(a_1, \ldots, a_n) = \rho(w) =\log^*(\Phi)(w) = \sum_{l=1}^n \frac{(-1)^l}{l} (\Phi \circ P)^{*l} (w), 
\end{equation}
where $P:= \id - u$ and $\Phi \circ P = \Phi - \epsilon$, such that $\Phi \circ P (\un)=0$ and $\Phi \circ P (w)=\Phi(w)$ for a word $w \neq \un$.


\subsection{Boolean cumulants as infinitesimal characters}
\label{ssec:booleancumulants}

Going back to Theorem \ref{thm:HA}, in \cite{ebrahimipatras_15} it was shown that the splitting of the coproduct \eqref{HA} 
\begin{equation}
\label{splitting}
	\Delta = \Delta^+_{\prec} + \Delta^+_{\succ},
\end{equation}
into the left half-unshuffle coproduct
\begin{equation}
\label{HAprec+}
	\Delta^+_{\prec}(a_1 \cdots a_n) := \sum_{1 \in S \subseteq [n]} a_S \otimes a_{J^S_{[n]}},
\end{equation}
and the right half-unshuffle coproduct
\begin{equation}
\label{HAsucc+}
	\Delta^+_{\succ}(a_1 \cdots a_n) := \sum_{1 \notin S \subset [n]} a_S \otimes a_{J^S_{[n]}},
\end{equation}
implies the next result.

\begin{thm} \cite{ebrahimipatras_15} \label{thm:bialg}
The Hopf algebra $H=\overline T(T(A))$ equipped with $\Delta^+_{\succ}$ and $\Delta^+_{\prec}$ is a unital unshuffle bialgebra. 
\end{thm}

Recall that unshuffle coalgebras and unshuffle bialgebras are defined as follows.
 
\begin{defn}
A counital unshuffle coalgebra is a coaugmented coalgebra $\overline C = C \oplus k.\un$ with coproduct
\begin{equation}
\label{codend}
	\Delta(c) := \bar\Delta(c) + c \otimes \un + \un \otimes c,
\end{equation}
such that on $C$, $\bar\Delta = \Delta_{\prec} + \Delta_{\succ}$ with 
\begin{eqnarray}
	(\Delta_{\prec} \otimes I) \circ \Delta_{\prec}   &=& (I \otimes \bar\Delta)\circ \Delta_{\prec}        	\label{C1}\\
  	(\Delta_{\succ} \otimes I) \circ \Delta_{\prec}   &=& (I \otimes \Delta_{\prec})\circ \Delta_{\succ} 	\label{C2}\\
   	(\bar\Delta \otimes  I) \circ \Delta_{\succ}         &=& (I \otimes \Delta_{\succ})\circ \Delta_{\succ}   \label{C3},
\end{eqnarray}
where $\Delta_{\prec}$ and $\Delta_{\succ}$ are called respectively left and right half-unshuffle coproducts.
\end{defn}

\begin{defn}
An unshuffle bialgebra is a unital and counital bialgebra $\overline B=B \oplus k.\un$ with product $\cdot_B$ and coproduct $\Delta$. At the same time $\overline B$ is a counital unshuffle coalgebra with $\bar\Delta = \Delta_{\prec} + \Delta_{\succ}$ such that the following compatibility relations hold 
\begin{eqnarray}
	\Delta^+_{\prec}(a \cdot_B b)  &=& \Delta^+_{\prec}(a)  \cdot \Delta(b)      	\label{D1}\\
  	\Delta^+_{\succ}(a \cdot_B b)  &=& \Delta^+_{\succ}(a)  \cdot \Delta(b),    	\label{D2}
\end{eqnarray}
where
\begin{eqnarray}
	\Delta^+_{\prec}(a)  &:=& \Delta_{\prec}(a) + a \otimes \un     	\label{D3}\\
  	\Delta^+_{\succ}(a)  &:=& \Delta_{\succ}(a) + \un \otimes a.     	\label{D4}
\end{eqnarray}
\end{defn}

Dualizing the coproduct splitting \eqref{splitting} leads to a splitting of the corresponding convolution product $f * g = m_k \circ (f \otimes g) \circ \Delta = f \succ g + f \prec g,$ where $f,g \in Lin(H^+,k)=Lin(T(T(A)),k)$, into left and right half-shuffle convolution products
$$
	f \prec g := m_k \circ (f\otimes g)\circ \Delta_\prec ,
\qquad\
	f \succ g := m_k \circ (f\otimes g)\circ \Delta_\succ .
$$
Half convolutions with the counit $\epsilon$ are defined by 
\begin{equation}
\label{unit-check}
	f\prec \epsilon =f=\epsilon\succ f,\ \ \epsilon\prec f=0=f\succ\epsilon .
\end{equation}

\begin{prop} \cite{ebrahimipatras_15}
The space ${\mathcal L}_A:=(Lin(H,k), \prec, \succ)$ is a unital shuffle algebra.
\end{prop}

Further below in Section \ref{sec:shuffle} we will recollect the definition and basic results on unital shuffle algebras. Here we recall that the splitting of the coproduct and the resulting shuffle products can be used to define additional shuffle-type exponentials, beside the shuffle exponential $\exp^*: g(A) \to G(A)$. Indeed, there exist two more bijections between $G(A)$ and $g(A)$. They are defined in terms of the left and right half-shuffle, or ``time-ordered'', exponentials 
$$
	\exp^{\prec}(\alpha) := \epsilon + \sum_{n > 0} \alpha^{\prec{n}} 
	\qquad\
	\exp^{\succ}(\alpha):= \epsilon + \sum_{n > 0}  \alpha^{\succ n},
$$
where $\alpha \in g(A)$, $\alpha^{\prec{n}}:=\alpha \prec (\alpha^{\prec{n-1}})$ and $\alpha^{\succ{n}} := (\alpha^{\succ{n-1}}) \succ \alpha$, with $\alpha^{\prec{0}} := \epsilon =: \alpha^{\succ{0}}$.

\begin{thm} \cite{ebrahimipatras_15}\label{thm:Gg}
Let $\Phi \in G(A)$. There exists a unique $\kappa  \in g(A)$ such that 
\begin{equation}
\label{recursion1}
	\Phi = \epsilon + \kappa\prec \Phi,
\end{equation}
and conversely, for $\kappa\in g(A)$ the map $\Phi:=\exp^{\prec}( \kappa)$ is a character. Analogously, for $\Psi \in G(A)$, there exists a unique $\alpha  \in g(A)$ such that 
\begin{equation}
\label{recursion2}
	\Psi = \epsilon + \Psi \succ \alpha ,
\end{equation}
and conversely, for $\alpha \in g(A)$ the map $\Psi:=\exp^{\succ}(\alpha)$ is a character. 
\end{thm}

\begin{thm} \cite{ebrahimipatras_15}\label{tim:freeprob}
Let $(A,\phi)$ be a non-commutative probability space with unital map $\phi: A \to k$ and $\Phi$ its extension to $\overline T(T(A))$ as a character. Let the map $\kappa: \overline T(T(A))\to k$ be the infinitesimal character solving $\Phi = \epsilon + \kappa\prec \Phi$. For $w=a_1 \cdots a_n  \in T(A)$ we set $k_n(a_1, \ldots ,a_n ):=\kappa(w)$, $m_n(a_1, \ldots, a_n ):=\Phi(w)$. Then $m_n(a_1, \ldots, a_n )$ and $k_n(a_1, \ldots, a_n )$ satisfy the moment-cumulant relation in free probability
\begin{equation}
\label{free}
	m_n(a_1, \ldots, a_n )=\sum_{\pi \in \mathcal{NC}_n} k_\pi(a_1, \ldots, a_n ).
\end{equation}
\end{thm}

The next theorem captures the moment-cumulant relation in the boolean case. See \cite{speicher_97c} for details on the boolean setting.

\begin{thm} \label{tim:boolean}
Let $(A,\phi)$ be a non-commutative probability space with unital map $\phi: A \to k$ and $\Phi$ its extension to $\overline T(T(A))$ as a character. Let the map $\beta: \overline T(T(A))\to k$ be the infinitesimal character solving $\Phi = \epsilon + \Phi \succ \beta$. For $w=a_1 \cdots a_n \in T(A)$  we set $r_n(a_1, \ldots, a_n ):=\beta(w)$, and $m_n(a_1, \ldots ,a_n ):=\Phi(w)$. Then
\begin{equation}
\label{boolean}
	m_n(a_1, \ldots , a_n )=\sum_{I \in \mathcal{B}_n} r_I(a_1, \ldots, a_n),
\end{equation}
where $\mathcal{B}_n$ is the boolean lattice of interval partitions. In particular, the $r_n(a_1, \ldots, a_n)$ identify with the (multivariate) boolean cumulants.
\end{thm}

\begin{proof}
The statement follows from the definition of $\Delta_{\succ}$ together with \eqref{unit-check}, which implies that for an arbitrary word $w=a_1 \cdots a_l \in T(A)$
$$
	 \Phi \succ \beta (w) = \sum_{j=1}^l \Phi(a_{j+1} \cdots a_l) \beta(a_1 \cdots a_{j}).
$$
This is due to the fact that $\beta$ is an infinitesimal character and therefore maps products as well as the unit to zero. Hence, we find that 
$$
	\Phi \succ \beta (w) = \sum_{j=1}^l m_{l-j}(a_{j+1}, \ldots, a_l) r_{j}(a_1, \ldots, a_{j}),
$$ 
which is the recurrence formula defining boolean cumulants in \cite{speicher_97c}. 
\end{proof}
	
So far we have shown that the three shuffle-type exponentials $\exp^*$, $\exp^{\prec}$ and $\exp^{\succ}$ defined on the  unital shuffle algebra ${\mathcal L}_A:=(Lin(H,k), \prec, \succ)$ respectively capture the moment-cumulant relations for monotone, free and boolean cumulants. In the monotone case the shuffle logarithm $\log^*$ permits to express monotone cumulants in terms of moments. In the next section we will use results from general shuffle algebra theory to derive the notions of left and right half-shuffle logarithm, which permit to express free and boolean cumulants in terms of the corresponding moments. Moreover, we will show that the three shuffle-type exponentials are tightly related. This will allow us to relate, using shuffle algebra tools and identities, the infinitesimal characters corresponding to monotone, free and boolean cumulants.


\section{Shuffle algebra, exponentials and half-shuffle fixed point equations}
\label{sec:shuffle}

We used in the previous section elementary properties of unshuffle coproducts and shuffle products. We will need later more advanced properties and identities that are the subject of the present section.


\subsection{Shuffle and pre-Lie algebras}
\label{sub:shufflepreLie}

The notion of half-shuffles and related identities first appeared in the work of Eilenberg and MacLane in topology in the 1950's. They used relations \eqref{A1}-\eqref{A3} below to obtain the first conceptual proof of the associativity of shuffle products in topology, at the chain level. They also introduced the axioms for commutative shuffle algebras (that appear in homology). Their ideas were rediscovered several times in different contexts, see e.g. \cite{fpijac}. One finds very often in the literature the operadic terminology ``dendriform algebras''  and ``Zinbiel algebras'' instead of shuffle and commutative shuffle algebras, because free shuffle algebras have basis parametrized by rooted trees, respectively because commutative shuffle algebras are dual in the operadic sense to the Bloh--Cuvier notion of Leibniz algebra. However, having a basis of free algebras parametrized by rooted trees is not a characteristic and meaningful property (it also holds for Lie algebras, magmas, pre-Lie algebras, among others). We prefer to use the terminology of ``shuffles'' which is more informative, better grounded in both history and applications of the theory, and conveys immediately the basic and fundamental underlying intuition, i.e., the splitting of an associative or commutative product into two ``half-shuffle products''.

A shuffle algebra consists of a $k$-vector space $D$ together with two bilinear compositions $\prec$ and $\succ$ (the left respectively right half-shuffle products) satisfying the shuffle relations
\begin{eqnarray}
	(a\prec b)\prec c  &=& a\prec(b \shuffle c)        	\label{A1}\\
  	(a\succ b)\prec c  &=& a\succ(b\prec c)   		\label{A2}\\
   	a\succ(b\succ c)   &=& (a \shuffle b)\succ c        	\label{A3}.
\end{eqnarray}
Left and right half-shuffles are neither commutative nor associative. However, the shuffle product $m_\shuffle: D \otimes D \to D$, $m_\shuffle (a \otimes b) =: a \shuffle b $ is defined as a linear combination of both half-shuffles
\begin{equation}
	 a \shuffle b := a \succ b + a \prec b. \label{dendassoc}       
\end{equation}
It is  non-commutative in general and from the shuffle relations it follows that \eqref{dendassoc} is associative. A commutative shuffle algebra is defined by adding the relation $a\prec b = b \succ a$ for all $a,b \in D$, which implies that \eqref{dendassoc} becomes a commutative product. 

A simple example of commutative shuffle algebra is defined on the algebra $F$ of smooth functions on $\mathbb{R}$ with pointwise product. The left and right half-shuffle products are given in terms of the indefinite Riemann integral 
\begin{equation*}
	(f \prec g)(\tau):=f(\tau)\int_0^\tau g(s)ds
	\qquad\
	(f \succ g)(\tau):=\int_0^\tau f(s)ds g(\tau).
\end{equation*}
The shuffle relations encode integration by parts. Note that $(f \prec g)(\tau) = (g \succ f)(\tau)$. However, the latter does not hold in case of matrix or operator valued functions.

Left and right pre-Lie algebras \cite{ChaLiv_01,manchon_11} can be defined on any shuffle algebra. Indeed, on $(D,\prec,\succ)$ the two products
\begin{equation}
\label{def:prelie}
    a \rhd b:= a \succ b - b \prec a
    \qquad\
    a \lhd b:= a \prec b - b\succ a
\end{equation}
respectively satisfy the left and right pre-Lie identities
\begin{eqnarray*}
    (a\rhd b)\rhd c-a\rhd(b\rhd c)&=& (b\rhd a)\rhd c-b\rhd(a\rhd c),\\
    (a\lhd b)\lhd c-a\lhd(b\lhd c)  &=& (a\lhd c)\lhd b-a\lhd(c\lhd b), 
\end{eqnarray*}
which define left pre-Lie respectively right pre-Lie algebras on $D$. The left pre-Lie identity can be rewritten in terms of the map $L_{a\rhd }:D \to D$, $b \mapsto L_{a\rhd }b:=a\rhd b$, such that $L_{[a,b] \rhd }=[L_{a\rhd },L_{b\rhd }]$. The bracket on the left-hand side is defined by $[a,b]:=a\rhd b-b\rhd a$ and satisfies the Jacobi identity. An analogous statement holds in the right pre-Lie case. It turns out that the commutators obtained from the products \eqref{dendassoc} and \eqref{def:prelie} all define the same Lie algebra structure on $D$.  

An augmented shuffle algebra $\overline D := D \oplus k.\un$ contains a unit $\un$, such that
\begin{equation}
\label{unit-dend}
    a \prec \un := a =: \un \succ a
    \qquad\ 
    \un \prec a := 0 =: a \succ \un,
\end{equation}
implying $a\shuffle\un=\un\shuffle a=a$. Note that by convention, $\un \shuffle \un=\un$, but $\un \prec \un$ and $\un \succ \un$ may not be defined consistently with respect to shuffle relations. See \cite{chapoton_02} for more details. We define $x^{\prec{n}}:=x \prec (x^{\prec{n-1}})$ and $x^{\succ{n}} := (x^{\succ{n-1}}) \succ x$, with $x^{\prec{0}} := \un=: x^{\succ{0}}$. Let $L_{a \succ}(b) := a \succ b$ and $R_{\succ b} (a):=a \succ b$, and similarly for $R_{\prec b}$ and $L_{\prec b}$. We then obtain that $L_{a \succ} L_{b \succ} = L_{a \shuffle b \succ}$, $R_{\prec a} R_{\prec b} = R_{\prec b \shuffle a}$, and $L_{a \succ} - R_{\prec a} = L_{a \rhd}$, $R_{\prec b} - L_{b \succ} = R_{\lhd b}$.


\subsection{Exponential maps}
\label{sub:expo}

A preliminary remark is in order regarding convergence issues: they are left aside in the present paper since we deal implicitely with formal series expansions over free shuffle algebras (insuring the convergence in the formal sense) or with graded algebras (in which case formal power series expansions restrict to finite expansions in each degree). In pratice, ``let $D$ be a shuffle algebra'' means therefore till the end of the present section, ``let $D$ be a free or a graded connected (i.e.~with no degree zero component) shuffle algebra''.

Under these assumptions, in a unital shuffle algebra $(\overline D,\prec,\succ)$ both the shuffle exponential and logarithm are defined for $a \in D$ in terms of the associative shuffle product~(\ref{dendassoc})
\begin{equation}
\label{ExpLog}
	\exp^\shuffle(a):=\un + \sum_{n > 0} \frac{a^{\shuffle n}}{n!}  
	\qquad\ 
	\log^\shuffle(\un+a):=-\sum_{n>0}(-1)^n\frac{a^{\shuffle n}}{n}. 
\end{equation}

For $a \in D$ the left and right half-shuffle, or ``time-ordered'', exponentials 
$$
	\exp^{\prec}(a) := \un + \sum_{n > 0} a^{\prec{n}} 
	\qquad\
	\exp^{\succ}(a):=\un + \sum_{n > 0}  a^{\succ n} 
$$
are defined as solutions $X=\exp^{\prec}(a)$ and $Z=\exp^{\succ}(a)$ in $\overline D$ of the linear left respectively  right half-shuffle fixed point equations
\begin{equation}
\label{recursion}
	X=\un + a\prec X \quad\ \quad\ Z= \un + Z \succ a.
\end{equation}
See \cite{ebrahimimanchon_09} for details and background. 

\begin{lem}\label{lem:inverse-shuffle}\cite{ebrahimipatras_15}
Let $D$ be a shuffle algebra, and $\overline D$ its augmentation by the unit $\un$. For $x \in D$, the solutions of half-shuffle fixed point equations, $X:=\exp^{\prec}(x)$ and $Y:= \exp^{\succ}(-x)$, satisfy $Y \shuffle X= X \shuffle Y = \un$. In particular, $Y=X^{-1}=\sum_{n\geq 0}(-1)^n(X- \un)^{\shuffle n}$.
\end{lem}

\begin{proof}
Indeed
\begin{eqnarray*}
	\exp^{\succ}(-x) \shuffle \exp^{\prec}(x)- \un
	&=& \sum\limits_{n+m\geq 1}(-1)^n\big\{(x^{\succ n})\prec (x^{\prec m}) 
					+ (x^{\succ n})\succ (x^{\prec m})\big\}\\
	&=& \sum\limits_{n>0,m\geq 0}(-1)^n(x^{\succ n})\prec (x^{\prec m}) 
				+ \sum\limits_{n\geq 0,m> 0}(-1)^n(x^{\succ n})\succ (x^{\prec m}).
\end{eqnarray*}
Now, since $(-1)^n(x^{\succ n})\prec (x^{\prec m})=(-1)^n((x^{\succ n-1})\succ x)\prec (x^{\prec m})=(-1)^n(x^{\succ n-1})\succ (x^{\prec m+1})$. That $X^{-1}=\sum_{n\geq 0}(-1)^n(X- \un)^{\shuffle n}$ follows from $\un = X \shuffle X^{-1}$ which implies that the shuffle inverse of $X$ has to satisfy the shuffle fixed point equation $X^{-1} = \un - (X- \un) \shuffle X^{-1}$. From this the proof follows. 
\end{proof}

A key property of half-shuffle exponentials is that they can be inverted using an appropriate half-shuffle notion of logarithm. These results are embodied in the following lemma and definition (see also~\cite{ebrahimipatras_15}).

\begin{lem}\label{lem:inverse}
Let $D$ be a shuffle algebra, and $\overline D$ its augmentation by the unit $\un$. For $a \in D$ and $X := \un + a \prec X = \exp^\prec (a)$ it follows that
\begin{equation}
\label{left-halfshuffleLOG}
	a = (X- \un) \prec X^{-1}.
\end{equation}
Analogously, for $Z = \exp^\succ (b)$ it follows that $b = Z^{-1} \succ (Z - \un)$.
\end{lem}

\begin{proof}
Identity \eqref{left-halfshuffleLOG} follows by applying $R_{\prec X^{-1}}$ to $(X- \un) = a \prec X$, thanks to shuffle relation \eqref{A1} and Lemma \ref{lem:inverse-shuffle}. 
\end{proof}

We may therefore define the corresponding left (right) half-shuffle logarithm.

 \begin{defn}\label{def:left-right-log}
Let $\overline D$ be an augmented shuffle algebra with unit $\un$. For $Y \in D$ we define the left half-shuffle logarithm, i.e., the inverse map to the left half-shuffle exponential
\begin{equation}
\label{left-log}
	  \log^{\prec}(\un + Y) := Y \prec \big(\sum\limits_{n \geq 0}(-1)^n Y^{\shuffle n}\big).
\end{equation}
The  right half-shuffle logarithm $ \log^{\succ}$ is defined analogously following Lemma \ref{lem:inverse}.
\end{defn}


\subsection{The pre-Lie Magnus expansion}
\label{sub:Magnus}

The following results will be essential in the context of expressing relations between the different cumulants. The solutions of the linear fixed point equations in \eqref{recursion} can be expressed using the shuffle exponential defined in \eqref{ExpLog} \cite{ebrahimimanchon_09}. More precisely, the solution $X=\exp^\prec (a)$ of $X=\un + a\prec X$ can be rewritten
\begin{equation}
\label{shuffleexp-solution}
	X = \exp^\shuffle\!\!\big(\Omega'(a)\big).
\end{equation}
The map $\Omega'$ is called pre-Lie Magnus expansion \cite{ebrahimimanchon_09} and obeys the following equation
\begin{equation}
\label{preLieMagnus}
	\Omega'(a) = \frac{L_{\Omega' \rhd}}{e^{L_{\Omega' \rhd}}-\un}(a)
            =\sum\limits_{m\ge 0} \frac{B_m}{m!}\ L^{(m)}_{\Omega' \rhd}(a)
            =a - \frac{1}{2} a \rhd a + \sum\limits_{m\ge 2} \frac{B_m}{m!}\ L^{(m)}_{\Omega' \rhd}(a).
\end{equation}
Here, the $B_l$'s are the Bernoulli numbers. By Lemma \ref{lem:inverse-shuffle}, the solution of $Z= \un + Z \succ a$ is given by $Z = \exp^\shuffle\!\!\big(-\Omega'(-a)\big).$ In the commutative case, i.e., when $a\succ b=b\prec a$ so that $a\rhd b=0$ for all $a,b \in D$, the map $\Omega'$ reduces to the identity map on $D$. Hence, in a commutative shuffle algebra the two fixed point equations in \eqref{recursion} coincide and the solution is given by $X= \exp^\shuffle(a).$ For $a \in D$ we define the map
\begin{equation}
\label{eq:W}
	W(a) := \frac{e^{L_{a \rhd}} - \un}{{L_{a \rhd}}}(a)= a + \frac 12 a\rhd a + \frac 16 a\rhd(a\rhd a) + \cdots.
\end{equation}
In general, i.e., for any pre-Lie algebra, one can show that $W$ is a bijection, and its compositional inverse $W^{-1} = \Omega'$. See \cite{manchon_11} for details.

The next lemma is a simple consequence of the identities $\exp^{\succ}(x) = \exp^\shuffle\!\!\big(-\Omega'(-x)\big)$ and $\exp^{\prec}(x) = \exp^\shuffle\!\!\big(\Omega'(x)\big)$ as well as $W \circ \Omega'(x) = x=\Omega' \circ W(x)$.  

\begin{lem}\label{lem:transforming}
Let $D$ be a shuffle algebra and $\overline D$ its augmentation by the unit $\un$. For $a \in D$ the following holds 
\begin{equation}
	\exp^{\prec}\big(W(a)\big)	= \exp^\shuffle\!\!\ (a)  
						= \exp^{\succ}\big(-W(-a)\big).  \label{transforming}
\end{equation}
\end{lem}

This result implies in particular that the solution of the left half-shuffle fixed point equation, $X = \un + a \prec X$, can be written using the right-half shuffle exponential, $X=\exp^{\succ}\big(-W(-\Omega'(a))\big)$, and vice versa, that is,  the solution of $Z = \un + Z \succ a$ can be described through the left-half shuffle exponential, $Z=\exp^{\prec}\big(W(-\Omega'(-a))\big)$.


\section{Relations among monotone, free and boolean cumulants}
\label{sec:relations}

Let us return to the recursions $\Phi=\epsilon + \kappa  \prec \Phi$ and $\Phi=\epsilon + \Phi \succ \beta$ in ${\mathcal L_A}$ where $\kappa,\beta \in g(A)$ are the infinitesimal characters corresponding to free respectively boolean cumulants. Recall that the convolution in ${\mathcal L_A}$ is a shuffle product in the sense of the foregoing section, i.e., $\Phi * \Psi = \Phi \succ \Psi + \Phi \prec \Psi$ for $\Phi, \Psi \in {\mathcal L_A}$. In the free case, where $\Phi=\epsilon + \kappa  \prec \Phi$, we can apply the left half-shuffle logarithm $\log^{\prec}(\Phi)$, and obtain
\begin{equation*}
	\kappa=\log^{\prec}(\epsilon  + (\Phi-\epsilon)) = (\Phi-\epsilon) \prec \Phi^{-1}.
\end{equation*}
Analogously, applying the right half-shuffle logarithm to $\Phi=\epsilon + \Phi \succ \beta$ yields $\beta=\log^{\succ}(\epsilon  + (\Phi-\epsilon))$. Since $\Phi$ is a character its inverse is $\Phi^{-1}=\Phi\circ S$, where $S$ is the antipode \eqref{antipode} of $\overline T(T(A))$.
In general, we have:

\begin{cor}
Let $\kappa $ be an infinitesimal character, such that the character $\Phi=\epsilon + \kappa \prec \Phi$ in ${\mathcal L_A}$ is given by $\Phi:=\exp^{\prec}(\kappa )$. Then by applying the left half-shuffle logarithm, $\kappa=\log^{\prec}(\epsilon  + (\Phi-\epsilon))$, we obtain
\begin{equation}
\label{FreeConv}
	\kappa  = (\Phi-\epsilon) \prec \Phi \circ S \in g(A).
\end{equation}
Analogously, for $\Psi = \exp^{\succ}(\alpha)$ it follows that by applying the right half-shuffle logarithm we obtain $\alpha = \log^{\succ}(\epsilon  + (\Psi-\epsilon))=\Psi \circ S \succ (\Psi-\epsilon) \in g(A)$.
\end{cor}

In the next theorem we apply Lemma \ref{lem:transforming}, which permits us to relate free, boolean and monotone cumulants seen as infinitesimal characters in $g(A)$.

\begin{thm}\label{thm:transforming}
Let $(A,\phi)$ be a non-commutative probability space with unital map $\phi: A \to k$ and $\Phi$ its multiplicative extension to $\overline T(T(A))$. Let $\kappa,\beta,\rho \in g(A)$ be the infinitesimal characters corresponding respectively to free, boolean and monotone cumulants. Then 
\begin{eqnarray}
	\Phi &=& \exp^\prec(\kappa) = \exp^*\!\big(\Omega'(\kappa)\big) = \exp^\succ\big(-W(-\Omega'(\kappa))\big) \label{freeTO}\\
	\Phi &=& \exp^\succ(\beta)  =\exp^*\!\big(-\Omega'(-\beta)\big) = \exp^\prec\big(W(-\Omega'(-\beta))\big)   \label{bolleanTO}\\
	\Phi &=& \exp^*(\rho)=\exp^\prec\big(W(\rho)\big) = \exp^\succ(-W(-\rho)\big) \label{monotoneTO}.
\end{eqnarray}
From this we deduce that $\rho =\Omega'(\kappa)$ and $\beta =-W(-\Omega'(\kappa))$ give monotone respectively boolean cumulants expressed in terms of free cumulants. And analogously for $\rho=-\Omega'(-\beta)$ and $\kappa=W(-\Omega'(-\beta))$, and $\kappa=W(\rho)$ and $\beta=-W(-\rho)$. 
\end{thm}

Theorem \ref{thm:transforming} can be used to recover the following formulas that appeared in \cite{lehner-etal_15}, and which relate multivariate monotone, free,  and boolean cumulants. Regarding equations \eqref{boolean-free} and \eqref{free-boolean} below, the reader is referred to the earlier references \cite{BelNic_08,Lehner_02}. From \cite{lehner-etal_15} we recall the notion of irreducible non-crossing partition, which is a non-crossing partition of the set $[n]$ with the first and last element $(1,n)$ being in the same block. The set of  irreducible non-crossing partitions is denoted by $\mathcal{NC}^{irr}_n$. Let $a_1, \ldots, a_n \in A$ and $r_n(a_1, \ldots, a_n)$, $k_n(a_1, \ldots, a_n)$, and $h_\pi(a_1, \ldots, a_n)$ denote the multivariate boolean, free and monotone cumulants, respectively.   
\begin{align}
\label{boolean-free}
	r_n(a_1, \ldots, a_n) &= \sum_{\pi \in \mathcal{NC}^{irr}_n} k_\pi (a_1, \ldots, a_n) \\
\label{free-boolean}
	k_n(a_1, \ldots, a_n) &= \sum_{\pi \in \mathcal{NC}^{irr}_n} (-1)^{|\pi|-1} r_\pi(a_1, \ldots, a_n)  \\
\label{boolean-monotone}
	r_n(a_1, \ldots, a_n) &= \sum_{\pi \in \mathcal{NC}^{irr}_n} \frac{1}{\tau(\pi)!} h_\pi(a_1, \ldots, a_n) \\
\label{free-monotone}
	k_n(a_1, \ldots, a_n) &= \sum_{\pi \in \mathcal{NC}^{irr}_n} \frac{(-1)^{|\pi|-1} }{\tau(\pi)!} h_\pi(a_1, \ldots , a_n). 
\end{align}

We derive these formulas by induction using Theorem \ref{thm:transforming}. Let $w=a_1 \cdots a_n \in T(A)$ and let $\kappa,\beta,\rho \in g(A)$ be the infinitesimal characters corresponding respectively to free, boolean and monotone cumulants. We  first address \eqref{boolean-free}, which is equivalent to 
$$
	\beta(a_1 \cdots a_n) = \sum_{1,n \in S \subset [n]} \kappa(a_S) \Phi(a_{J^S_{[n]}}), 
$$
with $\beta(a_1 \cdots a_n) = r_n(a_1, \ldots, a_n)$ and $\kappa(w) =k_n(a_1, \ldots, a_n)$. 
From $\exp^\succ(\beta)=\exp^\prec(\kappa)$, we deduce that $\Phi \succ \beta = \kappa \prec \Phi$, such that  
\begin{align}
	\Phi \succ \beta(w) &= \beta(w) + \sum_{j=1}^{n-1} \Phi(a_{j+1} \cdots a_n)\beta(a_1 \cdots a_j)
	=\sum_{1 \in S \subset [n]} \kappa(a_S) \Phi(a_{J^S_{[n]}}).
\end{align}
From this we obtain 
\begin{align}
	\beta(w) &=\sum_{1 \in S \subset [n]} \kappa(a_S) \Phi(a_{J^S_{[n]}})
		- \sum_{j=1}^{n-1} \Phi(a_{j+1} \cdots a_n)\beta(a_1 \cdots a_j)\\
		&= \sum_{1,n \in S \subset [n]} \kappa(a_S) \Phi(a_{J^S_{[n]}})
		 + \sum_{1 \in S \subset [n] \atop n \notin S} \kappa(a_S) \Phi(a_{J^S_{[n]}})
		 - \sum_{j=1}^{n-1} \Phi(a_{j+1} \cdots a_n)\beta(a_1 \cdots a_j) \label{calc1}.
\end{align}
It is clear that $\kappa(a)=\beta(a)$ for $a \in A$. For $n=2$ we find 
$$
	\beta(a_1a_2) = \kappa(a_1a_2) + \kappa(a_1)\kappa(a_2) - \Phi(a_2)\beta(a_1) = \kappa(a_1a_2) . 
$$
Hence, for $n>2$ we use induction and write $\beta(a_1 \cdots a_j)= \sum_{1,j \in S \subset [j]} \kappa(a_S) \Phi(a_{J^S_{[j]}})$ in \eqref{calc1}. Then
\begin{align}
	\beta(w) &= \sum_{1,n \in S \subset [n]} \kappa(a_S) \Phi(a_{J^S_{[n]}})
		 + \sum_{1 \in S \subset [n] \atop n \notin S} \kappa(a_S) \Phi(a_{J^S_{[n]}})\\
		 &- \sum_{j=1}^{n-1} \Phi(a_{j+1} \cdots a_n) \sum_{1,j \in T \subset [j]} \kappa(a_T) \Phi(a_{J^T_{[j]}})\\
		&=\sum_{1,n \in S \subset [n]} \kappa(a_S) \Phi(a_{J^S_{[n]}}).
\end{align}
In the last step we used that 
$$
	0= \sum_{1 \in S \subset [n] \atop n \notin S} \kappa(a_S) \Phi(a_{J^S_{[n]}})\\
		 - \sum_{j=1}^{n-1} \Big(\sum_{1,j \in T \subset [j]} \kappa(a_T) \Phi(a_{J^T_{[j]}})\Big)\Phi(a_{j+1} \cdots a_n).
$$

The inverse identity \eqref{free-boolean} follows similarly from $\Phi^{-1} \succ (-\kappa ) =  (-\beta) \prec \Phi^{-1}$, which is deduced from $\Phi^{-1} =\exp^\prec(- \beta)=\exp^\succ(-\kappa)$. 

Next we turn to identity \eqref{boolean-monotone}. It follows from using \eqref{monotone3}, i.e., the moment-cumulant relation in the monotone setting $\exp^*(\rho)(a_1\cdots a_n) = \sum_{\gamma \in \mathcal{NC}_n} \frac{1}{\tau(\gamma)!} h_{\gamma}(a_1,\ldots, a_n)$ and the right half-shuffle fixed point equation for boolean cumulants, $\Phi - \epsilon = \Phi \succ \beta$. We deduce via induction that 
\begin{align}
	\exp^*(\rho)(a_1\cdots a_n) &= \sum_{\gamma \in \mathcal{NC}_n} \frac{1}{\tau(\gamma)!} h_{\gamma}(a_1,\ldots, a_n) \\
	& =\sum_{j=1}^{n-1} \beta(a_1 \cdots a_j) \Phi(a_{j+1} \cdots a_n) + \beta(w)\\
	&= \sum_{j=1}^{n-1} \Big( \sum_{\sigma \in \mathcal{NC}^{irr}_j} \frac{1}{\tau(\sigma)!} h_{\sigma}(a_1, \ldots, a_j)\Big)\Phi(a_{j+1} \cdots a_n) + \beta(w)\\
	&= \sum_{j=1}^{n-1} \Big( \sum_{\sigma \in \mathcal{NC}^{irr}_j} \frac{1}{\tau(\sigma)!} h_{\sigma}(a_1 , \ldots, a_j)\Big)\Big(\sum_{\gamma \in \mathcal{NC}_{n-j}} \frac{1}{\tau(\gamma)!} h_{\gamma}(a_{j+1}, \ldots, a_n)\Big) + \beta(w).
\end{align} 
From this we conclude that 
\begin{align}
	\beta(w) &= \sum_{\gamma \in \mathcal{NC}_n} \frac{1}{\tau(\gamma)!} h_{\gamma}(w) 
	- \sum_{j=1}^{n-1} \Big( \sum_{\sigma \in \mathcal{NC}^{irr}_j} \frac{1}{\tau(\sigma)!} h_{\sigma}(a_1, \ldots, a_j)\Big)\Big(\sum_{\gamma \in \mathcal{NC}_{n-j}} \frac{1}{\tau(\gamma)!} h_{\gamma}(a_{j+1}, \ldots, a_n)\Big)\\
	&= \sum_{\sigma \in \mathcal{NC}^{irr}_n} \frac{1}{\tau(\sigma)!} h_{\sigma}(a_1, \ldots, a_n).
\end{align} 

The last identity \eqref{free-monotone} follows by a similar argument applied to $\Phi^{-1}-\epsilon = \Phi^{-1} \succ(-\kappa )$. The inverse relations of \eqref{boolean-monotone} and \eqref{free-monotone} are more involved as there is no fixed point equation for monotone cumulants available. Instead, one has to go back to Proposition \ref{thm:transforming}, from which we see that the inverse relation of \eqref{free-monotone} is given by 
$$
	h_n(a_1, \ldots,a_n )=\rho(a_1 \cdots a_n) =\Omega'(\kappa)(a_1 \cdots a_n) 
	= \sum_{n \ge 0} \frac{B_n}{n!} L^{(n)}_{\Omega'(\kappa) \rhd}(\kappa)(a_1 \cdots a_n).
$$
Note that the last sum has only finitely many terms. The inverse relation of \eqref{boolean-monotone} is given by 
$$ 
	h_n(a_1, \ldots,a_n )=\rho(a_1 \cdots a_n) =-\Omega'(-\beta)(a_1 \cdots a_n) 
	= -\sum_{n \ge 0} \frac{B_n}{n!} L^{(n)}_{\Omega'(-\beta) \rhd}(-\beta)(a_1 \cdots a_n) .
$$  
We leave the details of these calculations to the reader.


\end{document}